\documentclass[11pt]{article}
\usepackage{amsmath,amssymb,amsfonts,amsthm}
\usepackage{float}
\numberwithin{equation}{section}
\newtheorem{theorem}{Theorem}[section]
\newtheorem{definition}{Definition}[section]
\newtheorem{lemma}[theorem]{Lemma}

\begin{document}
\begin{center}
{\Large{\textbf{A note on the Brush Number of Jaco Graphs, $J_n(1), n \in \Bbb N$}}} 
\end{center}
\vspace{0.5cm}
\large{\centerline{(Johan Kok)\footnote {\textbf {Affiliation of author:}\\
\noindent Johan Kok (Tshwane Metropolitan Police Department), City of Tshwane, Republic of South Africa\\
e-mail: kokkiek2@tshwane.gov.za}}
\vspace{0.5cm}
\begin{abstract}
\noindent The concept of the brush number $b_r(G)$ was introduced for a simple connected undirected graph $G$. This note extends the concept to a special family of directed graphs and declares that the brush number $b_r(J_n(1))$ of a the finite Jaco graph, $J_n(1)), n \in \Bbb N$ with prime Jaconian vertex $v_i$ is given by:\\ \\
$b_r(J_n(1)) = \sum\limits_{j=1}^{i}(d^+(v_j) - d^-(v_j)) + \sum\limits_{j = i+1}^{n}max\{0, (n-j)- d^-(v_j)\}.$ \\ \\
\end{abstract}
\noindent {\footnotesize \textbf{AMS Classification Numbers:} 05C07, 05C12, 05C20, 05C38, 05C70} 
\section{Introduction}
\noindent The concept of the brush number $b_r(G)$ of a simple connected graph $G$ was introduced by McKeil [4] and Messinger et. al. [6]. The problem is initially set that all edges of a simple connected undirected graph $G$ is \emph{dirty}. A finite number of brushes, $\beta_G(v) \geq 0$ is allocated to each vertex $v \in V(G).$ Sequentially any vertex which has $\beta_G(v) \geq d(v)$ brushes allocated may send exactly one brush along a dirty edge and in doing so allocate an additional brush to the corresponding adjavent vertex (neighbour). The reduced graph $G' = G - vu_{\forall vu \in E(G)}$ is considered for the next iterative cleansing step. Note that a neighbour of vertex $v$ in $G$ say vertex $u$, now have $\beta_{G'}(u) =\beta_G(u) + 1.$\\ \\
\noindent Clearly for any simple connected undirected graph $G$ the first step of cleaning can begin if and only if at least one vertex $v$ is allocated, $\beta_G(v)= d(v)$ brushes. The minimum number of brushes that is required to allow the first step of cleaning to begin is, $\beta_G(u) = d(u) = \delta(G).$ Note that these conditions do not guarantee that the graph will be cleaned. The conditions merely assure at least the first step of cleaning.\\ \\
\noindent If a simple connected graph $G$ is orientated to become a directed graph, brushes may only clean along an out-arc from a vertex. Cleaning may initiate from a vertex $v$\ if and only if $\beta_G(v) \geq d^+(v)$ and $d^-(v) =0.$ The order in which vertices sequentially initiate cleaning is called the \emph{cleaning sequence} in respect of the orientation $\alpha_i$. The minimum number of brushes to be allocated to clean a graph for a given orientation $\alpha_i(G)$ is denoted $b_r^{\alpha_i}$. If an orientation $\alpha_i$ renders cleaning of the graph undoable we define $b_r^{\alpha_i} = \infty.$ An orientation $\alpha_i$ for which $b_r^{\alpha_i}$ is a minimum over all possible orientations is called \emph{optimal}.\\ \\
Now, since the graph $G$ on $\nu(G)$ vertices and having $\epsilon(G)$ edges can have $2^{\epsilon(G)}$ orientations, the \emph{optimal orientation} is not necessary unique. Let the set $\Bbb A =\{\alpha_i|\emph{ $\alpha_i$ an orientation of G}\}.$
\begin{lemma}
For a simple connected directed graph $G$, we have that:\\ \\ $b_r(G) = min_{\emph{over all $\alpha_i \in \Bbb A$}}(\sum_{v \in V(G)}max\{0, d^+(v) - d^-(v)\}) = min_{\forall \alpha_i}b_r^{\alpha_i}.$
\end{lemma}
\begin{proof}
See [7].
\end{proof}
\noindent Although we mainly deal with simple connected graphs it is easy to see that for set of simple connected graphs $\{G_1, G_2, G_3, ..., G_n\}$ we have that, $b_r(\cup_{\forall i}G_i) = \sum\limits_{i=1}^{n}b_r(G_i).$
\section{Brush Numbers of Jaco Graphs, $J_n(1), n \in \Bbb N$}
The infinite Jaco graph (\emph{order 1}) was introduced in $[2],$ and defined by $V(J_\infty(1)) = \{v_i| i \in \Bbb N\}$, $E(J_\infty(1)) \subseteq \{(v_i, v_j)| i, j \in \Bbb N, i< j\}$ and $(v_i,v_ j) \in E(J_\infty(1))$ if and only if $2i - d^-(v_i) \geq j.$\\ \\ The graph has four fundamental properties which are; $V(J_\infty(1)) = \{v_i|i \in \Bbb N\}$ and, if $v_j$ is the head of an edge (arc) then the tail is always a vertex $v_i, i<j$ and, if $v_k,$ for smallest $k \in \Bbb N$ is a tail vertex then all vertices $v_ \ell, k< \ell<j$ are tails of arcs to $v_j$ and finally, the degree of vertex $k$ is $d(v_k) = k.$ The family of finite directed graphs are those limited to $n \in \Bbb N$ vertices by lobbing off all vertices (and edges arcing to vertices) $v_t, t > n.$ Hence, trivially we have $d(v_i) \leq i$ for $i \in \Bbb N.$\\ \\
For ease of reference we repeat a few definitions found in [2].
\begin{definition}
The infinite Jaco Graph $J_\infty(1)$ is defined by $V(J_\infty(1)) = \{v_i| i \in \Bbb N\}$, $E(J_\infty(1)) \subseteq \{(v_i, v_j)| i, j \in \Bbb N, i< j\}$ and $(v_i,v_ j) \in E(J_\infty(1))$ if and only if $2i - d^-(v_i) \geq j.$
\end{definition}
\begin{definition}
The family of finite Jaco Graphs are defined by $\{J_n(1) \subseteq J_\infty(1)|n\in \Bbb {N}\}.$ A member of the family is referred to as the Jaco Graph, $J_n(1).$
\end{definition}
\begin{definition}
The set of vertices attaining degree $\Delta (J_n(1))$ is called the Jaconian vertices of the Jaco Graph $J_n(1),$ and denoted, $\Bbb{J}(J_n(1))$ or, $\Bbb{J}_n(1)$ for brevity.
\end{definition}
\begin{definition}
The lowest numbered (indiced) Jaconian vertex is called the prime Jaconian vertex of a Jaco Graph.
\end{definition}
\begin{definition}
If $v_i$ is the prime Jaconian vertex of a Jaco Graph $J_n(1)$, the complete subgraph on vertices $v_{i+1}, v_{i+2}, \cdots,v_n$ is called the Hope subgraph of a Jaco Graph and denoted,  $\Bbb{H}(J_n(1))$ or, $\Bbb{H}_n(1)$ for brevity.
\end{definition}
\noindent It is important to note that Definition 2.2 read together with Definition 2.1, prescribes a well-defined orientation of the underlying Jaco graph. So we have one defined orientation of the $2^{\epsilon(J_n(1))}$ possible orientations.
\begin{theorem}
For the finite Jaco Graph $J_n(1), n \in \Bbb N,$ with prime Jaconian vertex $v_i$ we have that:\\ \\
$b_r(J_n(1)) = \sum\limits_{j=1}^{i}(d^+(v_j) - d^-(v_j)) + \sum\limits_{j = i+1}^{n}max\{0, (n-j)- d^-(v_j)\}.$
\end{theorem}
\begin{proof}
Consider a Jaco Graph $J_n(1), n \in \Bbb N$ having the prime Jaconian vertex $v_i$. From the definition of a Jaco Graph (\emph{order 1}) it follows that $d^+(v_j) - d^-(v_j) \geq 0, 1\leq j \leq i.$ Hence, $max\{0, d^+(v_j) - d^-(v_j)\}_{1\leq j\leq i} = d^+(v_j) - d^-(v_j).$ So from Lemma 1.1 it follows that the first term must be $\sum\limits_{j=i}^{i}(d^+(v_j) - d^-(v_j))$ for the defined orientation.\\ \\
Similarly, it follows from the definition of a Jaco Graph that in the \emph{finite case}, $\ell = (n-j)- d^-(j), i+1 \leq j \leq n$ represents the shortage of brushes to initiate cleaning from vertex $v_j$ or, the surplus of brushes at $v_j$. Hence, $\ell > 0$ or $\ell \leq 0.$ So from Lemma 1.1 it follows that the second term must be $\sum\limits_{j=i+1}^{n}max\{0, (n-j) - d^-(v_j)\}$ for the defined orientation.\\ \\
So to settled the result we must show that no other orientation improves on the minimality of  $\sum\limits_{j=1}^{i}(d^+(v_j) - d^-(v_j)) + \sum\limits_{j=i+1}^{n}max\{0, (n-j) - d^-(v_j)\}.$\\ \\
Case 1: Consider the Jaco Graph, $J_1(1)$. Clearly be default, $b_r(J_1(1)) = 0.$\\ \\
Case 2: Consider the Jaco Graphs, $J_n(1), 2 \leq n\leq 4$. Label the edges of the underlying graph of $J_4(1)$, as $e_1 = v_1v_2, e_2 = v_2v_3, e_3 =v_3v_4.$  Now clearly, because we are considering paths, $P_2, P_3$ or $P_4$ only, the orientations $\{(v_1, v_2), (v_2, v_3), (v_3, v_4)\}$ or $\{(v_4, v_3), (v_3, v_2), (v_2, v_1)\}$ or respectively \emph{lesser} thereof, provide \emph{optimal orientations}. Thus the defined orientations of Jaco graphs, $J_n(1), 2 \leq n\leq 4$ are optimal. \\ \\
Case 3: Consider the Jaco Graph, $J_5(1)$. Label the edges of the underlying graph of $J_5(1)$ as $e_1 = v_1v_2, e_2 = v_2v_3, e_3 =v_3v_4, e_4 = v_3v_5, e_5 = v_4v_5.$ We know that $2^5$ cases need to be considered to exhaust all cases. Consider the orientations tabled below.\\ \\
\begin{tabular}{|c|c|c|c|c|}
\hline
$e_1$&$e_2$&$e_3$&$e_4$&$e_5$\\
\hline
$(v_1, v_2)$&$(v_2,v_3)$&$(v_3,v_4)$&$(v_3,v_5)$&$(v_4,v_5)$\\
\hline
$(v_1, v_2)$&$(v_2,v_3)$&$(v_3,v_4)$&$(v_3,v_5)$&$(v_5,v_4)$\\
\hline
$(v_1, v_2)$&$(v_2,v_3)$&$(v_3,v_4)$&$(v_5,v_3)$&$(v_4,v_5)$\\
\hline
$(v_1, v_2)$&$(v_2,v_3)$&$(v_3,v_4)$&$(v_5,v_3)$&$(v_5,v_4)$\\
\hline
$(v_1, v_2)$&$(v_2,v_3)$&$(v_4,v_3)$&$(v_3,v_5)$&$(v_4,v_5)$\\
\hline
$(v_1, v_2)$&$(v_2,v_3)$&$(v_4,v_3)$&$(v_3,v_5)$&$(v_5,v_4)$\\
\hline
$(v_1, v_2)$&$(v_2,v_3)$&$(v_4,v_3)$&$(v_5,v_3)$&$(v_4,v_5)$\\
\hline
$(v_1, v_2)$&$(v_2,v_3)$&$(v_4,v_3)$&$(v_5,v_3)$&$(v_5,v_4)$\\
\hline
$(v_1, v_2)$&$(v_3,v_2)$&$(v_3,v_4)$&$(v_3,v_5)$&$(v_4,v_5)$\\
\hline
$(v_1, v_2)$&$(v_3,v_2)$&$(v_3,v_4)$&$(v_3,v_5)$&$(v_5,v_4)$\\
\hline
$(v_1, v_2)$&$(v_3,v_2)$&$(v_3,v_4)$&$(v_5,v_3)$&$(v_4,v_5)$\\
\hline
$(v_1, v_2)$&$(v_3,v_2)$&$(v_3,v_4)$&$(v_5,v_3)$&$(v_5,v_4)$\\
\hline
$(v_1, v_2)$&$(v_3,v_2)$&$(v_4,v_3)$&$(v_3,v_5)$&$(v_4,v_5)$\\
\hline
$(v_1, v_2)$&$(v_3,v_2)$&$(v_4,v_3)$&$(v_3,v_5)$&$(v_5,v_4)$\\
\hline
$(v_1, v_2)$&$(v_3,v_2)$&$(v_4,v_3)$&$(v_5,v_3)$&$(v_4,v_5)$\\
\hline
$(v_1, v_2)$&$(v_3,v_2)$&$(v_4,v_3)$&$(v_5,v_3)$&$(v_5,v_4)$\\
\hline
\end{tabular}
\begin{tabular}{|c|c|c|c|c|}
\hline
$e_1$&$e_2$&$e_3$&$e_4$&$e_5$\\
\hline
$(v_2, v_1)$&$(v_2,v_3)$&$(v_3,v_4)$&$(v_3,v_5)$&$(v_4,v_5)$\\
\hline
$(v_2, v_1)$&$(v_2,v_3)$&$(v_3,v_4)$&$(v_3,v_5)$&$(v_5,v_4)$\\
\hline
$(v_2, v_1)$&$(v_2,v_3)$&$(v_3,v_4)$&$(v_5,v_3)$&$(v_4,v_5)$\\
\hline
$(v_2, v_1)$&$(v_2,v_3)$&$(v_3,v_4)$&$(v_5,v_3)$&$(v_5,v_4)$\\
\hline
$(v_2, v_1)$&$(v_2,v_3)$&$(v_4,v_3)$&$(v_3,v_5)$&$(v_4,v_5)$\\
\hline
$(v_2, v_1)$&$(v_2,v_3)$&$(v_4,v_3)$&$(v_3,v_5)$&$(v_5,v_4)$\\
\hline
$(v_2, v_1)$&$(v_2,v_3)$&$(v_4,v_3)$&$(v_5,v_3)$&$(v_4,v_5)$\\
\hline
$(v_2, v_1)$&$(v_2,v_3)$&$(v_4,v_3)$&$(v_5,v_3)$&$(v_5,v_4)$\\
\hline
$(v_2, v_1)$&$(v_3,v_2)$&$(v_3,v_4)$&$(v_3,v_5)$&$(v_4,v_5)$\\
\hline
$(v_2, v_1)$&$(v_3,v_2)$&$(v_3,v_4)$&$(v_3,v_5)$&$(v_5,v_4)$\\
\hline
$(v_2, v_1)$&$(v_3,v_2)$&$(v_3,v_4)$&$(v_5,v_3)$&$(v_4,v_5)$\\
\hline
$(v_2, v_1)$&$(v_3,v_2)$&$(v_3,v_4)$&$(v_5,v_3)$&$(v_5,v_4)$\\
\hline
$(v_2, v_1)$&$(v_3,v_2)$&$(v_4,v_3)$&$(v_3,v_5)$&$(v_4,v_5)$\\
\hline
$(v_2, v_1)$&$(v_3,v_2)$&$(v_4,v_3)$&$(v_3,v_5)$&$(v_5,v_4)$\\
\hline
$(v_2, v_1)$&$(v_3,v_2)$&$(v_4,v_3)$&$(v_5,v_3)$&$(v_4,v_5)$\\
\hline
$(v_2, v_1)$&$(v_3,v_2)$&$(v_4,v_3)$&$(v_5,v_3)$&$(v_5,v_4)$\\
\hline
\end{tabular}\\ \\ \\
For all possible orientations of $J_5(1)$ as tabled, we have: $b_r^{\alpha_1} = 2, b_r^{\alpha_2} = 2, b_r^{\alpha_3} = \infty, b_r^{\alpha_4} = 3, b_r^{\alpha_5} = 3, b_r^{\alpha_6} = \infty, b_r^{\alpha_7} = 3, b_r^{\alpha_8} = 3, b_r^{\alpha_9} = 4, b_r^{\alpha_{10}} = 4, b_r^{\alpha_{11}} = \infty, b_r^{\alpha_{12}} = 4, b_r^{\alpha_{13}} = 4, b_r^{\alpha_{14}} = \infty, b_r^{\alpha_{15}} = 3, b_r^{\alpha_{16}} = 3, b_r^{\alpha_{17}} = 3, b_r^{\alpha_{18}} = 3, b_r^{\alpha_{19}} = \infty, b_r^{\alpha_{20}} = 4, b_r^{\alpha_{21}} = 4, b_r^{\alpha_{22}} = \infty, b_r^{\alpha_{23}} = 3, b_r^{\alpha_{24}} = 4, b_r^{\alpha_{25}} = 3, b_r^{\alpha_{26}} = 3, b_r^{\alpha_{27}} = \infty, b_r^{\alpha_{28}} = 3, b_r^{\alpha_{29}} = 3, b_r^{\alpha_{30}} = \infty, b_r^{\alpha_{31}} = 2, b_r^{\alpha_{32}} = 2$\\ \\
It follows that the defined orientation of the Jaco graph $J_5(1),$ tabled as $\alpha_1$ has $b_r^{\alpha_1} =2 =  min_{\forall \alpha_i}b_r^{\alpha_i}.$ Since the prime Jaconian vertex of $J_5(1)$ is $v_3$, the result:\\
$\sum\limits_{j=1}^{3}(d^+(v_j) - d^-(v_j)) + \sum\limits_{j=4}^{5}max\{0, (5-j) - d^-(v_j)\},$ holds.\\ \\
Through induction assume the results holds for $J_k(1)$ having prime Jaconian vertex $v_i$. Consider the Jaco graph $J_{k+1}(1).$ Clearly $J_{k+1}(1) = J_k(1) + (v_j, v_{k+1})_{i+1\leq j \leq k}.$ So the minimum number of additional brushes to be added to the $b_r(J_k(1))$ brushes to clean $J_{k+1}(1)$ is given by $\sum\limits_{j=i+1}^{k}max\{0, d^+ (v_j) - d^-(v_j)\}_{\emph{in $J_{k+1}(1)$}}.$ So the minimum number of brushes to be allocated\\ \\ to clean $J_{k+1}$ is given by:\\ 
$b_r(J_{k+1}(1)) = \sum\limits_{j=1}^{i}(d^+(v_j) - d^-(v_j))_{\emph{in $J_k(1)$}} + \sum\limits_{j = i+1}^{k}max\{0, (k-j)- d^-(v_j)\}_{\emph{in $J_k(1)$}} +\\ \\ \sum\limits_{j=i+1}^{k}max\{0, d^+(v_j) - d^-(v_j)\}_{\emph{in $J_{k+1}(1)$}} =\\ \\
\sum\limits_{j=1}^{i+1}(d^+(v_j) - d^-(v_j))_{\emph{in $J_{k+1}(1)$}} + \sum\limits_{j = i+2}^{k+1}max\{0, ((k+1)-j)- d^-(v_j)\}_{\emph{in $J_{k+1}(1)$}}.$\\ \\ 
Since a re-orientation of any one, or more of the arcs $(v_j, v_{k+1})_{i+1\leq j \leq k}$ in $J_{k+1}(1)$ does not require more brushes, but could in some instances render the cleaning process undoable, the result holds in general.
\end{proof}
\noindent For illustration the adapted table below follows from the Fisher Algorithm $[2]$ for $J_n(1), n \in \Bbb N, n\leq 15.$ Note that the Fisher Algorithm determines $d^+(v_i)$ on the assumption that the Jaco Graph is always sufficiently large, so at least $J_n(1), n \geq i+ d^+(v_i).$ For a smaller graph the degree of vertex $v_i$ is given by $d(v_i)_{J_n(1)} = d^-(v_i) + (n-i).$ In $[2]$ Bettina's theorem describes an arguably, closed formula to determine $d^+(v_i)$. Since $d^-(v_i) = n - d^+(v_i)$ it is then easy to determine $d(v_i)_{J_n(1)}$ in a smaller graph $J_n(1), n< i + d^+(v_i).$\\ \\ \\ \\ \\ \\ \\ \\ \\
\noindent \textbf{Table 1}\\\\
\begin{tabular}{|c|c|c|c|c|}
\hline
$ i\in{\Bbb{N}}$&$d^-(v_i)$&$d^+(v_i)$&$v_j^*$&$b_r(J_i(1)), $\\
\hline
1&0&1&$v_1$&0\\
\hline
2&1&1&$v_1$&1\\
\hline
3&1&2&$v_2$&1\\
\hline
4&1&3&$v_2$&1\\
\hline
5&2&3&$v_3$&2\\
\hline
6&2&4&$v_3$&3\\
\hline
7&3&4&$v_4$&4\\
\hline
8&3&5&$v_5$&5\\
\hline
\end{tabular}
\begin{tabular}{|c|c|c|c|c|}
\hline
$ i\in{\Bbb{N}}$&$d^-(v_i)$&$d^+(v_i)$&$v_j^*$&$b_r(J_i(1)), $\\
\hline
9&3&6&$v_5$&6\\
\hline
10&4&6&$v_6$&7\\
\hline
11&4&7&$v_7$&8\\
\hline
12&4&8&$v_7$&9\\
\hline
13&5&8&$v_8$&11\\
\hline
14&5&9&$v_8$&12\\
\hline
15&6&9&$v_9$&14\\
\hline
16&6&10&$v_{10}$&16\\
\hline
\end{tabular}\\ \\
\noindent {\footnotesize{Vertex $v_j^*$ the prime Jaconian vertex.}}\\ \\  
\noindent From Theorem 2.1 and Lemma 1.1 the brush allocations can easily be determined. For example $J_9(1)$ requires the minimum brush allocations, $\beta_{J_9(1)}(v_1) = 1, \beta_{J_9(1)}(v_2) = 0, \beta_{J_9(1)}(v_3) = 1, \beta_{J_9(1)}(v_4) = 2, \beta_{J_9(1)}(v_5) = 1, \beta_{J_9(1)}(v_6) = 1, \beta_{J_9(1)}(v_7) = 0, \beta_{J_9(1)}(v_8) = 0, \beta_{J_9(1)}(v_9) = 0.$\\ \\
\noindent [Open problem:  It is known that for a complete graph $K_n, n \in \Bbb N$ we have $b_r(K_n) =n\lfloor\frac{n}{2}\rfloor - \lfloor\frac{n}{2}\rfloor^2 =\lfloor \frac{n^2}{4}\rfloor$, [5], [6]. Show that $b_r(J_n(1))_{n \in \Bbb N} \geq b_r(K_{n-i}) = (n-i)\lfloor\frac{n-i}{2}\rfloor - \lfloor\frac{n-i}{2}\rfloor^2$ if $v_i$ is the prime Jaconian vertex. In other words, $b_r(J_n(1)) \geq b_r(\Bbb H_n(1)).$]\\ \\
\noindent [Open problem: Consider the Jaco graph $J_n(1)$ with prime Jaconian vertex $v_i$. Seperate the Hope graph $K_{n-i}$ from $J_i(1)$ by removing the edges which link them. Let the number of edges which had to be removed be $\ell.$ What, if any, is the relationship between $b_r(J_n(1))$ and $\ell$?]\\ \\
\textbf{\emph{Open access:}} This paper is distributed under the terms of the Creative Commons Attribution License which permits any use, distribution and reproduction in any medium, provided the original author(s) and the source are credited. \\ \\
References (Limited) \\ \\
$[1]$  Bondy, J.A., Murty, U.S.R., \emph {Graph Theory with Applications,} Macmillan Press, London, (1976). \\
$[2]$ Kok, J., Fisher, P., Wilkens, B., Mabula, M., Mukungunugwa, V., \emph{Characteristics of Finite Jaco Graphs, $J_n(1), n \in \Bbb N$}, arXiv: 1404.0484v1 [math.CO], 2 April 2014. \\
$[3]$  Kok, J., Fisher, P., Wilkens, B., Mabula, M., Mukungunugwa, V., \emph{Characteristics of Jaco Graphs, $J_\infty(a), a \in \Bbb N$}, arXiv: 1404.1714v1 [math.CO], 7 April 2014. \\ 
$[4]$ McKeil, S., \emph{Chip firing cleaning process}, M.Sc. Thesis, Dalhousie University, (2007).\\ 
$[5]$ Messinger, M. E., \emph{Methods of decontaminating a network}, Ph.D. Thesis, Dalhousie University, (2008).\\
$[6]$ Messinger, M.E., Nowakowski, R.J., Pralat, P., \emph{Cleaning a network with brushes}. Theoretical Computer Science, Vol 399, (2008), 191-205.\\
$[7]$ Ta Sheng Tan,\emph{The Brush Number of the Two-Dimensional Torus,} arXiv: 1012.4634v1 [math.CO], 21 December 2010.
\end{document}